\documentclass[a4paper, 12pt, UKenglish]{article}

\usepackage{babel}
\usepackage[utf8]{inputenc}
\usepackage[T1]{fontenc}
\usepackage[a4paper]{geometry}
\usepackage{needspace}

\usepackage{libertine} % text font
\usepackage{inconsolata} % texttt font

\usepackage{amsmath, amsthm, amssymb}
\usepackage{graphicx}
\usepackage{enumerate}
\usepackage{authblk}
\usepackage{thmtools}
\usepackage{thm-restate}
\usepackage[colorlinks=true, citecolor=red]{hyperref}
\usepackage[capitalise]{cleveref}
\usepackage{float}
\usepackage{xspace}

\renewenvironment{abstract}
{\small\vspace{-1em}
\begin{center}
\bfseries\abstractname\vspace{-.5em}\vspace{0pt}
\end{center}
\list{}{
\setlength{\leftmargin}{0.6in}%
\setlength{\rightmargin}{\leftmargin}}%
\item\relax}
{\endlist}

\declaretheorem[name=Theorem, numberwithin=section]{theorem}
\declaretheorem[name=Lemma, sibling=theorem]{lemma}

\declaretheorem[name=Conjecture, sibling=theorem]{conjecture}

\declaretheorem[name=Claim, sibling=theorem]{claim}

\declaretheorem[name=Question, style=remark, sibling=theorem]{question}
\def\cqedsymbol{\hfill$\lrcorner$} 
\newcommand{\cqed}{\renewcommand{\qedsymbol}{\cqedsymbol}}

\newcommand\eqdef{\overset{\text{\tiny{def}}}{=}} 

\newcommand{\Cc}{\mathcal{C}}
\newcommand{\Pc}{\mathcal{P}}
\newcommand{\Nn}{\mathbb{N}}
\DeclareMathOperator{\mw}{mw}
\DeclareMathOperator{\fw}{fw}

\newcommand{\EH}{Erdős--Hajnal\xspace}

\let\le\leqslant
\let\ge\geqslant

\let\eps\varepsilon
\let\qle\preceq
\let\qge\succeq

\title{$\chi$-Boundedness and Neighbourhood Complexity \\ of Bounded Merge-Width Graphs}

\author[1]{Marthe Bonamy}
\author[2]{Colin Geniet\thanks{Supported by the Institute for Basic Science (IBS-R029-C1).}}
\affil[1]{CNRS, LaBRI, Université de Bordeaux, Bordeaux, France.}
\affil[2]{Discrete Mathematics Group, Institute for Basic Science (IBS), Daejeon, South Korea}
\date{\today}

\begin{document}

\maketitle

\begin{abstract}
    Merge-width, recently introduced by Dreier and Toruńczyk, is a common generalisation of bounded expansion classes and twin-width for which the first-order model checking problem remains tractable.
    We prove that a number of basic properties shared by bounded expansion and bounded twin-width graphs also hold for bounded merge-width graphs:
    they are $\chi$-bounded, they satisfy the strong \EH property, and their neighbourhood complexity is linear.
\end{abstract}

% edition
%\overfullrule=50pt % to spot overfull hbox

\section{Introduction}\label{sec:intro}
Dreier and Toruńczyk introduced merge-width in~\cite{merge-width} as the next step in the program of characterising the graph classes that have fixed parameter tractable algorithms for first-order model checking---the problem of testing a given first-order formula~$\phi$ on a given graph~$G$.
Bounded merge-width classes admit such an algorithm (assuming appropriate witnesses of merge-width are given),
and encompass both bounded expansion classes and bounded twin-width classes,
thus unifying the known model checking algorithms for the latter two~\cite{dvorak2013expansion,twin-width1}.
Merge-width is conjectured to be equivalent to \emph{flip-width}, previously introduced by Toruńczyk with similar goals~\cite{flip-width}, but for which the model checking problem remains open.

In this work, we show with elementary proofs that merge-width satisfies a few classical graph theoretic properties, the most significant being \emph{$\chi$-boundedness}.
This generalises on the one hand the simple remark that bounded expansion graphs are degenerate, and thus have bounded chromatic number,
and on the other hand the fact that bounded clique-width~\cite{dvorak2012cliquewidth} and bounded twin-width~\cite{twin-width3} graphs are $\chi$-bounded.
We leave open the question of \emph{polynomial} $\chi$-boundedness, which in the case of clique-width and twin-width required the introduction of much more involved tools~\cite{bonamy2020cliquewidth,pilipczuk2023quasipoly,bourneuf2025bounded}.

Merge-width is defined through \emph{merge sequences}:
for a graph~$G$, a merge sequence consists of $(\Pc_1,R_1),\dots,(\Pc_m,R_m)$ where:
\begin{enumerate}
  \item Each~$\Pc_i$ is a partition of~$V(G)$, with $\Pc_1 = \{\{x\} : x \in V(G)\}$ being the partition into singletons, $\Pc_m = \{V(G)\}$ being the trivial partition, and each~$\Pc_i$ being coarser than (or equal to)~$\Pc_{i-1}$, meaning that each part~$P \in \Pc_i$ is obtained by merging any number of parts of~$\Pc_{i-1}$.
  \item $R_1 \subseteq \dots \subseteq R_m \subseteq \binom{V(G)}{2}$ is a monotone sequence of sets of \emph{resolved pairs}.
  \item For any two (possibly equal) parts~$A,B \in \Pc_i$, the pairs in $AB \setminus R_i$ (i.e.\ the \emph{unresolved pairs} between~$A$ and~$B$) are either all edges, or all non-edges in~$G$.
\end{enumerate}
To restrict merge sequences, one defines their \emph{width}, parametrised by a radius~$r \in \Nn$:
it is the maximum over all steps~$i \ge 2$ and vertices~$v \in V(G)$ of the number of parts in~$\Pc_{i-1}$ accessible from~$v$ by a path of length at most~$r$ in the graph $(V(G),R_i)$ of resolved pairs.
The mismatch of indices between~$\Pc_{i-1}$ and~$R_i$ is intentional, and prevents one from simultaneously merging too many parts and adding too many resolved pairs when going from~$(\Pc_{i-1},R_{i-1})$ to~$(\Pc_i,R_i)$.
The \emph{radius-$r$ merge-width} of~$G$, denoted by~$\mw_r(G)$, is the minimum radius-$r$ width of a merge sequence for~$G$.
A class~$\Cc$ of graphs has \emph{bounded merge-width} if there is a function~$f$ such that any~$G \in \Cc$ satisfies~$\mw_r(G) \le f(r)$ for all~$r$.

We will show that bounded merge-width classes have the strong \EH property, are $\chi$-bounded, and have linear neighbourhood complexity.
All three properties were previously known to hold in bounded expansion classes and bounded twin-width classes, and the first two generalisations were conjectured in~\cite[Section~1, Discussion]{merge-width}.
Linear neighbourhood complexity also generalises to bounded flip-width classes.
We show that it is enough to bound merge-width (or flip-width) at radius~1 or~2 to obtain these results.

At the core of our proofs is a technique sometimes called \emph{freezing}, previously used in the context of twin-width~\cite{twin-width3,gajarsky2022stable}.
Freezing reduces the study of a merge-sequence~$\Pc_1,\dots,\Pc_m$ to a single, judiciously chosen partition.
In its simplest form, this means looking at the first (or last) step~$i$ for which~$\Pc_i$ satisfies the desired property.
The more general form of freezing creates a new partition~$\Pc$ across~$\Pc_1,\dots,\Pc_m$,
by considering the inclusion-wise maximal parts satisfying the property of interest.
That is, for each vertex~$v$, one considers the maximum~$i$ such that the part~$P \in \Pc_i$ containing~$v$ satisfies the desired property, and~$\Pc$ is defined as the collection of all such~$P$.

\subsection{Strong \EH property}
The strong \EH property refers to the presence of linear-size bicliques or anti-bicliques.
Bounded twin-width classes were shown to satisfy this property in~\cite[Theorem~22]{twin-width3} with a very simple proof.
For bounded expansion classes, it is an immediate corollary of degeneracy.
We show in \cref{sec:EH} that the same holds for bounded radius-1 merge-width.
\begin{restatable}{theorem}{erdoshajnal}
  \label{thm:mw-erdos-hajnal}
  Any graph~$G$ with~$n$ vertices satisfying~$\mw_1(G) = k$ contains disjoint subsets~$A,B$ of size~$\Omega(\frac{n}{k^2})$ such that~$A,B$ are either complete or anti-complete.
\end{restatable}
Our proof is inspired by the one for twin-width: we consider the first step~$i$ in the merge sequence such that~$\Pc_i$ has a part of size at least~$\eps n$, for a well chosen~$\eps$.

\subsection[χ-boundedness]{$\chi$-boundedness}
Next, we prove in \cref{sec:chi} that any class~$\Cc$ of graphs with bounded merge-width is $\chi$-bounded:
the chromatic number~$\chi(G)$ is bounded by some function of the size of a maximum clique~$\omega(G)$.
Precisely:
\begin{restatable}{theorem}{mwchibounded}\label{thm:chi-bounded}
    Any graph~$G$ with~$\mw_2(G) \le k$ and~$\omega(G) \le t$ satisfies
    \[ \chi(G) \le (t+1)!k^{2t-2}. \]
\end{restatable}
Bounded expansion classes are $\chi$-bounded in a trivial sense: they are degenerate, and thus have bounded chromatic number (and clique number).
Bounded twin-width classes were shown to be $\chi$-bounded in \cite[Theorem~21]{twin-width3}.
This was improved to reach polynomial $\chi$-boundedness in \cite{pilipczuk2023quasipoly,bourneuf2025bounded}, i.e.\ when twin-width is fixed, the chromatic number is bounded by a polynomial function of the clique number.
Dreier and Toruńczyk also asked whether the same holds for bounded merge-width.
\begin{conjecture}[\cite{merge-width}]\label{conj:poly-chi-bounded}
    Bounded merge-width classes are polynomially $\chi$-bounded.
\end{conjecture}

In our proof of \cref{thm:chi-bounded} we use merge-width at radius~2.
It is unclear to us whether this is necessary, or radius~1 can suffice.
\begin{question}\label{qst:radius1-chibounded}
    Is the class $\{ G : \mw_1(G) \le k \}$, with $k$ arbitrary, $\chi$-bounded?
\end{question}
Specifically, radius-2 merge-width is used in \cref{lem:struct-bounded-colouring}, while the remainder of our proof only requires a bound on radius-1 merge-width.
It may be that \cref{lem:struct-bounded-colouring} can be improved or modified to answer \cref{qst:radius1-chibounded} positively.

Conversely, one may answer it negatively by constructing a class~$\Cc$ of graphs with bounded radius-1 merge-width and bounded clique number (possibly even triangle-free), but unbounded chromatic number.
Such a class~$\Cc$ needs to contain arbitrarily large bicliques: indeed graphs with bounded radius-1 merge-width and no biclique~$K_{t,t}$ as subgraph are degenerate \cite[Corollary~7.7]{merge-width}, and thus have bounded chromatic number.
On the other hand, $\Cc$ having bounded radius-1 merge-width requires it to have bounded symmetric difference \cite[Lemma~7.20]{merge-width}, meaning that graphs in~$\Cc$ and all their induced subgraphs must contain a pair of vertices whose neighbourhoods differ only on a bounded size set.
The only examples of non $\chi$-bounded graph classes with bounded symmetric difference and containing arbitrarily large bicliques we are aware of
are shift graphs~\cite{erdos1968chromatic} which do not have bounded radius-1 merge-width~\cite{edouard}, and twincut graphs~\cite{bonnet2023tamed}.

It is also natural to ask whether \cref{thm:chi-bounded} can be generalised from merge-width to flip-width.
The definition of flip-width, which is based on cops and robber games, however seems poorly suited to the study of $\chi$-boundedness.
We believe that the most reasonable approach to this question is to prove the conjecture that merge-width and flip-width are equivalent \cite[Conjecture~1.17]{merge-width}.

\subsection{Neighbourhood complexity}
The \emph{neighbourhood complexity} function~$\pi_G(p)$ of a graph~$G$ is defined as the maximum number of distinct neighbourhoods over a set of~$p$ vertices in~$G$, that is
\[ \pi_G(p) \eqdef \max_{\substack{X \subset V(G) \\ |X| = p}} \#\{N(v) \cap X : v \not\in X\}. \]
This definition extends to a class~$\Cc$ of graphs as $\pi_\Cc(p) \eqdef \max_{G \in \Cc} \pi_G(p)$.
In general, this function can be exponential.
Dreier and Toruńczyk noted that any graph~$G$ with $\mw_1(G) \le k$ has near-twins, i.e.\ vertices whose neighbourhoods differ by at most~$2k$ \cite[Lemma~7.20]{merge-width}.
It follows that they have bounded VC-dimension, which by the Sauer--Shelah lemma gives a polynomial bound on their neighbourhood complexity.

We show in \cref{sec:nghbd} that it is even linear when merge-width at radius~2 is bounded.% EOL commented out to avoid line skip
\begin{restatable}{theorem}{mwneighbourhood}\label{thm:mw-linear-density}
    Any graph~$G$ with $\mw_2(G) \le k$ has neighbourhood complexity
    \[ \pi_G(p) \le k2^{k+2} \cdot p. \]
\end{restatable}
Radius~2 is optimal in this result:
$k$-degenerate graphs have bounded radius-1 merge-width~\cite[Theorem~7.3]{merge-width} but can have neighbourhood complexity~$\Theta(p^k)$.

Linear neighbourhood complexity of bounded expansion classes was established in~\cite{reidl2019boundedexpansion}.
For bounded twin-width graphs, it was proved independently by~\cite{tww-kernels} and~\cite{tww-VC-density}, and the bound was significantly improved in~\cite{bonnet2024neighbourhood}.

The proof of \cref{thm:mw-linear-density} uses a density increase argument to find dense obstructions to linear neighbourhood complexity, from which we derive a lower bound on merge-width.
Using the same technique together with an obstruction to flip-width called \emph{hideouts}~\cite{flip-width}, we can also prove \cref{thm:mw-linear-density} for flip-width.
\begin{restatable}{theorem}{fwneighbourhood}\label{thm:fw-linear-density}
    Any graph~$G$ with $\fw_2(G) \le k$ has neighbourhood complexity
    \[ \pi_G(p) \le 2^{2k+1} \cdot p. \]
\end{restatable}
Once again, bounding flip-width at radius~2 is optimal in this result, as degenerate graphs have bounded radius-1 flip-width \cite[Theorem~4.4]{flip-width}.

Reidl, Villaamil, and Stavropoulos proved not only that bounded expansion classes have linear neighbourhood complexity,
but also that having linear neighbourhood complexity at radius~$r$ (replacing neighbourhoods by balls of radius~$r$) for all~$r$ characterises bounded expansion among subgraph-closed graph classes~\cite{reidl2019boundedexpansion}.
In an insightful footnote, they suggest that dropping the `subgraph-closed' condition may lead to an interesting notion.
Considering only the balls of radius~$r$ however is insufficient for dense graphs, as they typically have diameter~2, rendering the condition meaningless beyond~$r=1$.
The correct generalisation in dense graphs uses first-order transductions.
\begin{conjecture}\label{conj:mw-linear-characterisation}
    A class~$\Cc$ has bounded merge-width if and only if every first-order transduction of~$\Cc$ has linear neighbourhood complexity.
\end{conjecture}
Since bounded merge-width is stable under first-order transductions \cite[Theorem~1.12]{merge-width}, \cref{thm:mw-linear-density} proves the left-to-right implication.
Remark that proving \cref{conj:mw-linear-characterisation} would imply that flip-width and merge-width are equivalent as conjectured in \cite[Conjecture~1.17]{merge-width}, since flip-width also has linear neighbourhood complexity and is closed under first-order transduction.
Naturally, one may first ask the same question with merge-width replaced by flip-width.

On the other hand, the second half of \cite[Conjecture~1.17]{merge-width}, claiming that merge-width is the \emph{dense analogue} of bounded expansion, implies \cref{conj:mw-linear-characterisation}~\cite{szymon}.

\section{Preliminaries}\label{sec:prelim}
We work with simple undirected graphs $G = (V,E)$.
The vertex and edge sets are also denoted as $V(G) = V$ and $E(G) = E$.
The neighbourhood~$N(x)$ of a vertex~$x \in V(G)$ is the set of vertices adjacent to~$x$.
The ball~$B^r_G(x)$ of radius~$r$ around~$x$ is the set of vertices connected to~$x$ by a path of length at most~$r$.

A subset of vertices $X \subset V$ is a clique if all pairs of vertices in~$X$ are edges; it is an independent set if none of them are.
The maximum size of a clique in~$G$, called \emph{clique number}, is denoted by~$\omega(G)$.
The graph consisting of a clique on~$t$ vertices is denoted as~$K_t$,
and a graph~$G$ with~$\omega(G) < t$ is said to be \emph{$K_t$-free}.
Two disjoint subsets of vertices $A,B \subset V$ are called \emph{complete}, resp.\ \emph{anti-complete}, if~$ab$ is an edge, resp.\ non-edge, for all pairs~$a \in A$, $b \in B$.

A proper $k$-colouring is a map $c : V(G) \to \{1,\dots,k\}$ assigning different values to adjacent vertices.
Equivalently, it is a partition of~$V(G)$ into~$k$ independent sets.
The graph~$G$ is \emph{$k$-colourable} if it admits a $k$-colouring.
The chromatic number of~$G$, denoted~$\chi(G)$, is the smallest~$k$ such that~$G$ is $k$-colourable.

A graph~$G$ is \emph{$k$-degenerate} if it admits an ordering~$<$ of its vertices such that each vertex has at most~$k$ neighbours to its right, i.e.\ $\#\{y \in N(x) : y > x\} \le k$ for all vertices~$x$.
All $k$-degenerate graphs are $(k+1)$-colourable through a greedy procedure.

Given a partition~$\Pc$ of the vertices of a graph~$G$, we define the quotient graph~$G/\Pc$
where~$\Pc$ is the set of vertices, and two parts $X,Y \in \Pc$ are adjacent if and only if there exists some edge $xy \in E(G)$ with $x \in X$, $y \in Y$.
When each part~$X \in \Pc$ is an independent set in~$G$, any $k$-colouring of the quotient~$G/\Pc$ gives a $k$-colouring of~$G$.

\paragraph{Merge-width}
Given a merge sequence $(\Pc_1,R_1),\dots,(\Pc_m,R_m)$,
a pair of vertices $xy \in R_i$ is called a \emph{resolved pair in~$R_i$}, or simply a resolved pair when~$i$ is clear from the context.
Conversely, a pair $xy \not\in R_i$ is said to be \emph{unresolved in~$R_i$}.
If~$xy$ is a resolved pair and is also an edge, then it is called a resolved edge.
We similarly talk about resolved or unresolved edges or non-edges.

In one instance, we will use a minimality assumption on merge sequences:
the merge sequence $(\Pc_1,R_1),\dots,(\Pc_m,R_m)$ is \emph{minimal} for~$G$ if the resolved sets~$R_i$ are inclusion-wise minimal for this choice of partitions~$\Pc_i$,
that is, there does not exist a different sequence of resolved sets $R'_1,\dots,R'_m$ such that $(\Pc_1,R'_1),\dots,(\Pc_m,R'_m)$ is a valid merge sequence for~$G$, and $R'_i \subseteq R_i$ for all~$i$.
Clearly, any merge sequence can be turned into a minimal merge sequence without increasing its width at any radius.
We use the following property of minimal merge sequences.
\begin{lemma}
  \label{lem:minimal-merge-sequence}
  In a minimal merge sequence $(\Pc_1,R_1),\dots,(\Pc_m,R_m)$ for~$G$,
  consider vertices $x_1,x_2 \in X$ and $y_1,y_2 \in Y$ for some parts $X,Y \in \Pc_i$,
  such that the pairs~$x_1y_1$ and~$x_2y_2$ are both unresolved in~$R_i$.
  Then $x_1y_1,x_2y_2$ are either both resolved or both unresolved in~$R_j$ for all~$j$.
\end{lemma}
\begin{proof}
  Observe first that the definition of merge sequence requires $x_1y_1,x_2y_2$ to be either both edges or both non-edges.
  Without loss of generality, let us assume that they are edges.

  For~$j \le i$, the claim is trivial as $x_1y_1,x_2y_2$ are not in~$R_i$, and $R_1 \subseteq \dots \subseteq R_m$ is monotone.
  Assume for a contradiction that for some~$j > i$, $x_1y_1$ is resolved in~$R_j$ but~$x_2y_2$ is not.

  Consider any step~$\ell$ with $i < \ell \le j$, and call~$X',Y'$ the parts containing~$X,Y$ in~$\Pc_\ell$.
  Since~$x_2y_2 \not\in R_j$, we a fortiori have~$x_2y_2 \not\in R_\ell$, i.e.\ $x_2y_2$ is an unresolved edge in~$R_\ell$.
  Thus, all unresolved pairs in~$X'Y' \setminus R_\ell$ must be edges, and removing the edge~$x_1y_1$ from~$R_\ell$ (if it were there) does not break this requirement.
  Define thus $R'_\ell = R_\ell \setminus \{x_1y_1\}$ for all $i < \ell \le j$, and $R'_\ell = R_\ell$ otherwise, so that $R'_1 \subseteq \dots \subseteq R'_m$ is monotone.
  Then $(\Pc_1,R'_1),\dots,(\Pc_m,R'_m)$ is a new valid merge sequence for~$G$ in which we removed~$x_1y_1$ from~$R_j$, and did not add any new resolved pair to any~$R_\ell$, contradicting the minimality of~$R_1,\dots,R_m$.
\end{proof}
One may observe that the conclusion of \cref{lem:minimal-merge-sequence} is always satisfied by \emph{construction sequences},
presented in \cite[Section~1]{merge-width} as an alternative definition of merge-width.

\paragraph{Flip-width}
A \emph{$k$-flip} of~$G$ is a graph~$G'$ obtained by picking a partition~$\Pc$ of~$V(G)$ into at most~$k$ parts,
and for each pair of parts~$A,B \in \Pc$ (including~$A=B$), choosing whether or not to \emph{flip} all pairs in~$A \times B$,
i.e.\ replacing edges with non-edges and vice-versa.
This implies that the adjacency matrix of~$G'$ is obtained from that of~$G$ by adding modulo~2 a matrix with rank at most~$k$.

Flip-width, denoted $\fw_r(G)$ for flip-width at radius~$r$,
is defined by a cops-and-robber game in which the robber moves at speed~$r$, and the cops can perform a $k$-flip instead of simply occupying vertices~\cite{flip-width}.
We will not use the definition of flip-width, and instead rely on an obstruction called \emph{hideouts} defined in~\cite[section~5.2]{flip-width}.
An \emph{$(r,k,d)$-hideout} in a graph~$G$ is a subset~$U$ of vertices satisfying the following:
for any $k$-flip~$G'$ of~$G$, the set $\{v \in U : |B_{G'}^r(v) \cap U| \le d\}$ of vertices of~$U$ with few distance-$r$ neighbours in~$U$ itself has size at most~$d$.
\begin{lemma}[{\cite[Lemma~5.16]{flip-width}}]\label{lem:hideout}
    If~$G$ contains an $(r,k,d)$-hideout for some~$d \in \Nn$, then~$\fw_r(G) > k$.
\end{lemma}

\section{Strong \EH property}\label{sec:EH}
\erdoshajnal*

The following key lemma will be applied to the graph of resolved pairs.
\begin{lemma}
  \label{lem:bipartite-resolved}
  Consider a bipartite graph~$(U,V,E)$ with $|U| = m$ and $|V| = n$, together with a partition~$\Pc$ of~$V$ in which no part has size more than~$\frac{n}{2k}$.
  Suppose that each vertex~$u \in U$ is adjacent to fewer than~$k$ parts of~$\Pc$.
  Then there are anti-complete sets $A \subseteq U$, $B \subseteq V$ of sizes $|A| \ge \frac{m}{k}$ and $|B| \ge \frac{n}{2k}$.
\end{lemma}
\begin{proof}
  We proceed by induction on~$k$, the base case~$k=1$ being trivial as~$E$ is empty.

  Since parts of~$\Pc$ have size at most~$\frac{n}{2k}$, we can pick a subset of parts of~$\Pc$ whose union~$B \subset V$ has size at least~$\frac{n}{2k}$, and less than~$\frac{n}{k}$.
  Denote by~$A \subseteq U$ the vertices with no neighbours in~$B$.
  If $|A| \ge \frac{m}{k}$, then~$A,B$ is the desired pair and we are done.
  When that is not the case, define $U' = U \setminus A$ and $V' = V \setminus B$.
  Their sizes $m' \eqdef |U'|$ and $n' \eqdef |V'|$ satisfy 
  \begin{equation}
    \frac{m'}{k-1} \ge \frac{m}{k} \qquad \text{and} \qquad \frac{n'}{k-1} \ge \frac{n}{k}.
  \end{equation}
  Finally, any vertex~$u \in U'$ is adjacent to some part of~$\Pc$ contained in~$B$, hence~$u$ is adjacent to fewer than~$k-1$ parts contained in~$V'$.
  We conclude by applying the induction hypothesis to~$U',V'$.
\end{proof}

\begin{proof}[Proof of \cref{thm:mw-erdos-hajnal}]
  In a merge sequence $(\Pc_1,R_1),\dots,(\Pc_m,R_m)$ for~$G$ with radius-1 width~$k$,
  pick the maximal step~$i$ such that all parts in~$\Pc_{i-1}$ have size at most~$\eps n$, for
  \begin{equation}
    \eps \eqdef \frac{1}{2k+4}
  \end{equation}
  Thus there is a part~$P \in \Pc_i$ with size more than~$\eps n$.
  Choose~$U \subseteq P$ of size~$\lceil \eps n \rceil$ arbitrarily.
  Fix $V' = V(G) \setminus U$. We will use the gross lower bound $|V'| \ge (1-2\eps)n$.

  \begin{claim}
    There are subsets~$A \subseteq U$ and $B \subseteq V'$ such that all pairs in~$AB$ are unresolved,
    and with sizes $|A| \ge \frac{n}{2(k+1)(k+2)}$ and $|B| \ge \eps n$.
  \end{claim}
  \begin{proof}
    Consider the bipartite graph $(U,V',R_i)$, and the partition~$\Pc_{i-1}$ restricted to~$V'$.
    By definition of radius-1 merge-width, each vertex of~$U$ is adjacent (in the sense of~$R_i$) to fewer than~$k+1$ parts.
    Notice that $\eps = \frac{1-2\eps}{2k+2}$, and thus parts of~$\Pc_{i-1}$ have size at most
    \begin{equation}
      \eps n = \frac{(1-2\eps) n}{2k+2} \le \frac{|V'|}{2(k+1)}.
    \end{equation}
    Thus \cref{lem:bipartite-resolved} yields the desired sets~$A,B$ with sizes
    \[
      |A| \ge \frac{\eps n}{k+1} = \frac{n}{2(k+1)(k+2)}
      \qquad \text{and} \qquad
      |B| \ge \frac{(1-2\eps)n}{2(k+1)} = \eps n. \cqed\qedhere
    \]
  \end{proof}

  Note that~$A$ is contained within the part~$P \in \Pc_i$.
  By contrast, $B$ might be spread across arbitrarily many parts of~$\Pc_i$.
  Given~$b \in B$, consider the part~$P' \in \Pc_i$ containing~$b$ (which might be~$P$ itself).
  Since the pairs~$ab$ for~$a \in A$ are all unresolved in~$R_i$ and all between~$P,P' \in \Pc_i$, they are either all edges, or all non-edges.
  That is, any given~$b \in B$ is either fully adjacent or fully non-adjacent to~$A$.
  By pigeonhole principle, we find~$B' \subset B$ of size at least~$|B|/2$ such that~$A,B'$ are complete or anti-complete.
  Since~$|B|$ is much larger than~$2|A|$, the size of~$A$ is the limiting factor, and we obtain the strong \EH bound
  \[ |A|,|B'| \ge \frac{n}{2(k+1)(k+2)}. \qedhere \]
\end{proof}

\section[χ-boundedness]{$\chi$-boundedness}\label{sec:chi}
\mwchibounded*

Say that a merge sequence $(\Pc_1,R_1),\dots,(\Pc_k,R_K)$ is \emph{structurally $\omega$-bounded} if for any part $P \in \Pc_i$ which does not induce an independent set in~$G$, all edges incident to a vertex of~$P$ are resolved pairs in~$R_i$.
In particular, there can never be an unresolved edge between two vertices of the same part.
If such a sequence exists with radius-1 width~$k$, then~$G$ has no $(k+1)$-clique.
Indeed, if~$X$ induces a clique, consider the first step~$\Pc_i$ in which two vertices~$u,v \in X$ are in the same part.
Then all edges incident to~$u$ are resolved in~$R_i$, i.e.\ all of~$X$ is within distance~1 of~$u$ in~$R_i$.
Since the vertices of~$X$ are all in distinct parts of~$\Pc_{i-1}$, this implies that the radius-1 width is at least~$|X|$.
We first prove our result under this assumption.

\begin{lemma}\label{lem:struct-bounded-colouring}
    Graphs with a structurally $\omega$-bounded merge sequence of radius-2 width~$k$ are $k$-colourable.
\end{lemma}
\begin{proof}
    Consider a merge sequence $(\Pc_1,R_1),\dots,(\Pc_m,R_m)$ for~$G$ subject to all the conditions in the statement.

    We say that a part~$P \in \Pc_i$ is \emph{resolved at step~$i$} if all edges of~$G$ incident to~$P$ are resolved in~$R_i$.
    Otherwise, the part~$P$ is \emph{unresolved at step~$i$}, and the assumption that the merge sequence is structurally $\omega$-bounded gives that~$P$ induces an independent set.
    Say that~$P \in \Pc_i$ is \emph{maximally unresolved at step~$i$} if it is unresolved, and for any~$j > i$, the part of~$\Pc_j$ containing~$P$ is resolved.
    In particular, if~$P$ is maximally unresolved at step~$i$, then all edges incident to~$P$ are present in~$R_{i+1}$.
    The collection of maximally unresolved parts is a partition~$\Pc$ of~$V(G)$.

    \begin{claim}\label{claim:part-radius-1}
        For any maximally unresolved~$P \in \Pc_i$, there is a vertex~$x \in V(G)$ such that $xy \in R_{i+1}$ is a resolved pair for all~$y \in P$.
    \end{claim}
    \begin{proof}
        By assumption, $P$ is unresolved at step~$i$, i.e.\ there is an edge $xy \in E(G)$ with $y \in P$ which is not present in~$R_i$.
        Then the definition of merge sequence requires that for any~$y' \in P$, either~$xy'$ is present in~$R_i$, or it is an edge in~$E(G)$, which must thus be present in~$R_{i+1}$ by maximality of~$i$.
        Either way~$xy' \in R_{i+1}$ for all~$y' \in P$. \cqed
    \end{proof}

    Define the index of $P \in \Pc$ as the maximal step~$i$ at which~$P$ is unresolved.
    We order~$\Pc$ by indices, that is we define the quasi-ordering~$\qle$ by $P \qle Q$ if~$i,j$ are the indices of~$P,Q$ respectively, and $i \le j$.
    We claim that each part~$P \in \Pc$ is adjacent to fewer than~$k$ other parts~$Q \in \Pc$ with~$Q \qge P$.

    Consider the vertex~$x$ given by \cref{claim:part-radius-1} for~$P$.
    Suppose that~$yz$ is an edge in~$G$ with $y \in P$ and $z \in Q$.
    Then we have~$xy \in R_{i+1}$ by choice of~$x$, and~$yz \in R_{i+1}$ since~$P$ is maximally unresolved, hence~$xyz$ is a path of length~2 from~$x$ to~$Q$ in~$(V,R_{i+1})$.
    Now applying the definition of width of the merge sequence, consider the at most~$k$ parts~$Q'_1,\dots,Q'_k$ of~$\Pc_i$ within distance~2 of~$x$ in~$(V,R_{i+1})$.
    Notice that~$P$ itself is one of these parts due to the pair~$xy \in R_{i+1}$, say~$Q'_k = P$.
    Since~$Q \qge P$, $Q$ comes from a partition~$\Pc_j$ for some~$j \ge i$, implying that~$Q$ is a union of parts of~$\Pc_i$.
    It follows that~$Q$ contains one of~$Q'_1,\dots,Q'_{k-1}$ (but not~$Q'_k = P$ since~$P,Q$ are distinct parts of~$\Pc$).
    Thus there cannot be more than~$k-1$ such parts~$Q$ in~$\Pc$.

    This proves that the quotient graph~$G/\Pc$ is $(k-1)$-degenerate.
    Since each part of~$\Pc$ induces an independent set in~$G$, this implies that~$G$ is $k$-colourable.
\end{proof}

Our goal is now to reduce the general case to that of structurally $\omega$-bounded merge sequences.

\begin{lemma}\label{lem:edge-partition}
    Let~$G = (V,E)$ be a graph with $\mw_2(G) \le k$ and $\omega(G) = t$.
    Then~$G$ can be edge-partitioned as $E = E_R \cup E_U \cup E_I$ such that:
    \begin{enumerate}
        \item $G_I \eqdef (V,E_I)$ is a disjoint union of induced subgraphs of~$G$, and satisfies $\omega(G_I) < t$,
        \item $G_R \eqdef (V,E_R)$ has a structurally $\omega$-bounded merge sequence of radius-2 width~$k$,
        \item and $G_U \eqdef (V,E_U)$ is $(kt+1)$-colourable.
    \end{enumerate}
\end{lemma}

Let us first quickly show that \cref{lem:edge-partition} implies the main result of this section.
\begin{proof}[Proof of \cref{thm:chi-bounded}]
    Consider~$G$ with~$\mw_2(G) \le k$ and~$\omega(G) = t$. We prove
    \begin{equation}
      \chi(G) \le (t+1)!k^{2t-2}
    \end{equation}
    by induction on~$t$, the base case~$t = 1$ (i.e.~$G$ being edgeless) being trivial.

    By \cref{lem:edge-partition}, $G$ has an edge-partition into~$G_I,G_R,G_U$, satisfying the following:
    \begin{enumerate}
        \item $G_I$ is a disjoint union of induced subgraphs of~$G$ and satisfies $\omega(G_I) < t$.
        Merge-width cannot increase when taking induced subgraphs or making disjoint unions, hence~$\mw_2(G_I) \le k$.
        By induction hypothesis, this gives $\chi(G_I) \le t!k^{2t-4}$.
        \item $G_R$ has a structurally $\omega$-bounded merge sequence of radius-2 width~$k$.
        Thus \cref{lem:struct-bounded-colouring} gives $\chi(G_R) \le k$.
        \item $G_U$ is $(kt+1)$-colourable.
    \end{enumerate}
    The chromatic number of an edge union of graphs is bounded by the product of their chromatic numbers, thus we have as desired
    \begin{align*}
        \chi(G) & \le \chi(G_I) \cdot \chi(G_R) \cdot \chi(G_U) \\
        & \le t!k^{2t-4} \cdot k \cdot (kt+1) \\
        & \le (t+1)!k^{2t-2}. \qedhere
    \end{align*}
\end{proof}

\begin{proof}[Proof of \cref{lem:edge-partition}]
    Consider a minimal merge sequence $(\Pc_1,R_1),\dots,(\Pc_m,R_m)$ with radius-2 width~$k$.

    Say that~$P \in \Pc_i$ is maximally $K_t$-free if~$P$ does not contain a clique~$K_t$, but the part of~$\Pc_{i+1}$ containing~$P$ does.
    The family of maximally $K_t$-free parts is a partition~$\Pc$ of~$V(G)$.
    For~$P \in \Pc$, we call \emph{index} of~$P$ the step~$i$ such that~$P$ is maximally $K_t$-free in~$\Pc_i$.
    We order~$\Pc$ by indices, i.e.\ $P \qle Q$ if and only if the index of~$P$ is less than or equal to that of~$Q$.

    We define a set~$R$ of `resolved pairs' for~$\Pc$.
    Note that~$(\Pc,R)$ will not satisfy the requirement that the unresolved pairs between two parts be all edges or all non-edges.
    Consider a pair~$xy$ belonging to parts~$x \in P$ and~$y \in Q$, and let~$i$ be the minimum of the indices of~$P$ and~$Q$.
    Then we choose to have $xy \in R$ if and only if $xy \in R_{i+1}$.
    The edge partition of~$G$ is then defined as follows:
    \begin{enumerate}
        \item $E_I$ consists of all edges between two vertices of the same part of~$\Pc$,
        \item $E_R = (E(G) \cap R) \setminus E_I$ contains resolved edges between distinct parts of~$\Pc$, and
        \item $E_U = E(G) \setminus (R \cup E_I)$ contains unresolved edges between distinct parts of~$\Pc$.
    \end{enumerate}
    By construction, each part~$P \in \Pc$ induces a $K_t$-free subgraph of~$G$, thus the requirement on~$E_I$ is satisfied.

    \begin{claim}
        For each~$P \in \Pc$, there are at most~$kt$ other parts~$Q \qge P$ adjacent to~$P$ in~$E_U$.
    \end{claim}
    \begin{proof}
        Call~$P'$ the part of~$\Pc_{i+1}$ containing~$P$. By assumption, $P'$ contains a clique~$X$ on~$t$ vertices.
        Assume that there is an edge $uv \in E_U$ with $u \in P$, $v \in Q$, and $P \qle Q$.
        By definition of~$E_U$, $uv$ is an unresolved edge in~$R_{i+1}$.

        There must be some~$x \in X$ such that either~$xv$ is a non-edge or~$x = v$, as otherwise $X \cup \{v\}$ would be a $(t+1)$-clique in~$G$.
        If~$xv$ is a non-edge, then it must be resolved, as we cannot have both the unresolved edge~$uv$ and the unresolved non-edge~$xv$ between~$v$ and~$P'$.
        Thus either~$xv \in R_{i+1}$ or~$x = v$.
        Either way, $v$ belongs to one of the at most~$k$ parts of~$\Pc_i$ which are within distance~1 of~$x$ in $(V,R_{i+1})$.
        Since~$Q \in \Pc_j$ for some~$j \ge i$, it follows that~$Q$ fully contains one of these at most~$k$ parts.
        Multiplying by the~$t$ choices of~$x \in X$, this leaves at most~$kt$ parts~$Q \qge P$ adjacent to~$P$ in~$E_U$. \cqed
    \end{proof}
    Thus in $G_U = (V,E_U)$, each part of~$\Pc$ induces an independent set, and the quotient graph~$G_U/\Pc$ is $kt$-degenerate.
    It follows that~$G_U$ is $(kt+1)$-colourable as desired.

    Finally, we consider the graph~$G_R = (V,E_R)$.
    \begin{claim}\label{clm:merge-sequence-valid}
        The merge sequence $(\Pc_1,R_1),\dots,(\Pc_m,R_m)$ for~$G$ is also a valid merge sequence for~$G_R$.
    \end{claim}
    \begin{proof}
        Consider parts~$A,B \in \Pc_i$, and two unresolved pairs $uv, u'v' \in AB \setminus R_i$.
        Assuming that~$uv$ is an edge in~$E_R$, we need to show that~$u'v'$ is also an edge.
        Recall that~$E_R = (E(G) \cap R) \setminus E_I$.

        Call~$A',B'$ the parts of~$\Pc$ containing~$u,v$ respectively, and~$j$ the minimum of their indices.
        We have $uv \in E_R \subseteq R$, hence $uv \in R_{j+1}$ by definition of~$R$.
        Thus $uv \not\in R_i$ but $uv \in R_{j+1}$, implying~$j \ge i$.
        It follows that $A \subseteq A'$ and $B \subseteq B'$.
        In particular, $u,u' \in A'$, $v,v' \in B'$, and~$A' \neq B'$ as otherwise~$uv$ would be in~$E_I$ and not~$E_R$.

        Next, since~$uv$ is an edge in~$E_R$ and a fortiori in~$E(G)$, the conditions on the merge sequence $(\Pc_1,R_1),\dots,(\Pc_m,R_m)$ for~$G$ give that~$u'v'$ is also an edge in~$E(G)$.

        Finally, since we have $u,u' \in A$, $v,v' \in B$, and the pairs~$uv,u'v'$ are both unresolved in~$R_i$, and since we assumed the merge sequence $(\Pc_1,R_1),\dots,(\Pc_m,R_m)$ to be minimal, \cref{lem:minimal-merge-sequence} gives that~$uv,u'v'$ are either both resolved or both unresolved in~$R_{j+1}$.
        We know that~$uv \in R_{j+1}$, thus $u'v' \in R_{j+1}$ too, which implies $u'v' \in R$.
        Thus the pair~$u'v'$ is an edge of~$E(G)$, is resolved in~$R$, and does not belong to~$E_I$, proving $u'v' \in E_R$ as desired. \cqed
    \end{proof}

    Additionally, the width of the merge sequence $(\Pc_1,R_1),\dots,(\Pc_m,R_m)$ is the same whether it is seen as a merge sequence for~$G$ or for~$G_R$:
    indeed the width is defined only in terms of the partitions~$\Pc_i$ and the resolved pairs~$R_i$.
    It finally remains to check that the merge sequence $(\Pc_1,R_1),\dots,(\Pc_m,R_m)$ is structurally $\omega$-bounded for~$G_R$.

    Suppose that~$uv \in E_R \setminus R_i$ is an unresolved edge between parts~$A,B \in \Pc_i$.
    As already argued in the proof of \cref{clm:merge-sequence-valid}, this implies that the parts~$A',B' \in \Pc$ containing~$u$ and~$v$ respectively have indices at least~$i$, and in particular~$A \subseteq A'$, $B \subseteq B'$.
    By choice of~$E_R$, parts of~$\Pc$ induce independent sets in~$G_R$, and a fortiori so do~$A$ and~$B$.

    Thus $(\Pc_1,R_1),\dots,(\Pc_m,R_m)$ is a structurally $\omega$-bounded merge sequence for~$G_R$, and its radius-2 width is at most~$k$, as desired.
\end{proof}

\section{Neighbourhood complexity}\label{sec:nghbd}
We will use the following lemma to ensure a form of density in obstructions to linear neighbourhood complexity.
For a pair of vertices~$x,y$, we write $\triangle(x,y) \eqdef N(x) \triangle N(y)$ for the symmetric difference of their neighbourhoods.
\begin{lemma}\label{lem:density-increase}
    If~$G$ has neighbourhood complexity~$\pi_G(p)$ not bounded by~$\alpha p$ for a constant~$\alpha$,
    then there are disjoint subsets~$X,Y$ of vertices such that
    \begin{enumerate}
        \item vertices in~$Y$ have pairwise distinct neighbourhoods in~$X$, and
        \item any~$x,x' \in X$ have neighbourhoods differing on more than~$\alpha$ vertices of~$Y$, i.e.\ $|Y \cap \triangle(x,x')| > \alpha$.
    \end{enumerate}
\end{lemma}
\begin{proof}
    By assumption, there exist disjoint sets~$X,Y$ with sizes $|Y| > \alpha |X|$ such that vertices in~$Y$ have pairwise distinct neighbourhoods in~$X$.
    Choose~$X$ minimal so that such a~$Y$ exists, and assume for a contradiction that~$x,x' \in X$ satisfy $|Y \cap \triangle(x,x')| \le \alpha$.
    Consider then $X' \eqdef X \setminus \{y\}$ and $Y' \eqdef Y \setminus \triangle(x,x')$.
    They still satisfy $|Y'| > \alpha |X'|$, and vertices in~$Y'$ have pairwise distinct neighbourhoods over~$X'$, contradicting the minimality of~$X$.
\end{proof}

\mwneighbourhood*
\begin{proof}
    Assume for a contradiction that the neighbourhood complexity of~$G$ is more than~$k2^{k+2} p$,
    and consider the disjoint subsets~$X,Y$ given by \cref{lem:density-increase}:
    vertices in~$Y$ have pairwise distinct neighbourhoods in~$X$,
    while vertices in~$X$ pairwise have neighbourhoods differing on more than~$k2^{k+2}$ vertices of~$Y$.

    Fix a merge sequence $(\Pc_1,R_1),\dots,(\Pc_m,R_m)$ with radius-2 merge-width equal to~$k$,
    and define~$\Pc_i$ to be the first step in which two vertices of~$X$, say~$x,y$, belong to the same part.
    Define~$A = Y \cap \triangle(x,y)$, which by choice of~$X,Y$ satisfies~$|A| > k2^{k+2}$.
    For any~$a \in A$, $ax$ and~$ay$ are neither both edges nor non-edges. Thus one of them must be resolved in~$R_i$.
    Up to symmetry, assume that~$ax$ is the resolved pair in~$R_i$ at least half of the time,
    and define $B \eqdef \{a \in A : ax \in R_i\}$, so that $|B| > k2^{k+1}$.

    All of~$B$ is within distance~1 of~$x$ in~$R_i$, hence it must intersect at most~$k$ parts of~$\Pc_{i-1}$, and a fortiori of~$\Pc_i$.
    Thus one of them, say~$P \in \Pc_i$, satisfies $|B \cap P| > 2^{k+1}$. Define $C \eqdef B \cap P$.
    \begin{claim}
        There are at least~$k+1$ vertices in~$X$ that are neither fully adjacent nor fully non-adjacent to~$C$.
    \end{claim}
    \begin{proof}
        Consider the adjacency matrix of~$X$ versus~$C$ as a matrix over~$\mathbb{F}_2$.
        Vertices in~$C \subseteq Y$ have distinct neighbourhoods over~$X$ and~$|C| > 2^{k+1}$,
        thus this matrix has rank at least~$k+2$.
        This implies that there are at least~$k+2$ vertices $x_1,\dots,x_{k+2} \in X$ with non-null and pairwise distinct neighbourhoods over~$C$.
        At most one of them is fully adjacent to~$C$.
        This leaves at least~$k+1$ vertices in~$X$ which are neither fully adjacent nor fully non-adjacent to~$C$.
        \cqed
    \end{proof}
    To conclude, remark that since~$C$ is contained within the part~$P \in \Pc_i$,
    each of these~$k+1$ vertices $x_1,\dots,x_{k+1}$ must be part of a resolved pair~$cx_j \in R_i$ for some~$c \in C$.
    This implies that all~$x_j$ are within distance~$2$ of~$x$ in~$R_i$.
    However, the vertices~$x_i \in X$ are all in distinct parts of~$\Pc_{i-1}$ by choice of~$\Pc_i$,
    contradicting the assumption that $(\Pc_1,R_1),\dots,(\Pc_m,R_m)$ has radius-2 width~$k$.
\end{proof}

\subsection{Flip-width}
\fwneighbourhood*

Recall that an \emph{$(r,k,d)$-hideout} in a graph~$G$ is a subset~$U$ of vertices such that in any $k$-flip~$G'$ of~$G$,
for all but~$d$ vertices $x \in U$, the distance-$r$ neighbourhood $B_{G'}^r(x)$ contains more than~$d$ vertices of~$U$.
We will show that the sets provided by \cref{lem:density-increase} yield a $(2,k,k)$-hideout, which implies that radius-2 flip-width is more than~$k$ by \cref{lem:hideout}.
\begin{proof}
    Using \cref{lem:density-increase}, consider subsets $X,Y \subset V(G)$ such that
    vertices in~$Y$ have pairwise distinct neighbourhoods in~$X$,
    while vertices in~$X$ pairwise have neighbourhoods differing on more than~$2^{2k+1}$ vertices of~$Y$.
    Let us prove that~$X$ is a $(2,k,k)$-hideout.

    Consider a $k$-flip~$G'$ of~$G$ obtained through a partition~$\Pc$ of~$V(G)$ into~$k$ parts.
    Define~$X'$ as the set of vertices~$x \in X$ satisfying $|B^2_{G'}(x) \cap X| \le k$,
    and assume for a contradiction that~$X'$ has size at least~$k+1$.
    Two vertices~$x,y \in X'$ must belong to the same part of~$\Pc$.
    Then the symmetric difference of~$N(x)$ and~$N(y)$ is exactly the same in~$G$ and~$G'$,
    and in particular $|Y \cap \triangle_{G'}(x,y)| > 2^{2k+1}$,
    which implies that one of the two, say~$x$, is adjacent to more than~$2^{2k}$ vertices of~$Y$.

    Define~$A = N_{G'}(x) \cap Y$, and consider the adjacency matrices~$M$ of~$G$ and~$M'$ of~$G'$
    (as matrices over the 2-element field~$\mathbb{F}_2$).
    Since vertices in~$Y$ have pairwise distinct neighbourhoods over~$X$ in~$G$,
    the restriction of~$M$ to rows in $A \subseteq Y$ and columns in~$X$ has rank more than~$2k$.
    Now by definition of a $k$-flip, $M$ and~$M'$ differ by a matrix of rank at most~$k$.
    It follows that the same~$A \times X$ submatrix in~$M'$ still has rank more than~$k$.
    This implies that in~$G'$, $A$ is adjacent to more than~$k$ vertices of~$X$.
    These more than~$k$ vertices are in~$X \cap B^2_{G'}(x)$, a contradiction.
\end{proof}

\section*{Acknowledgements}
The authors would like to thank Szymon Toruńczyk for insightful discussions, notably regarding \cref{conj:mw-linear-characterisation},
and Édouard Bonnet for helpful remarks about \cref{qst:radius1-chibounded}.

\bibliographystyle{alphaurl}
\bibliography{refs}

\end{document}